\documentclass[12pt,a4paper]{amsart}
\usepackage{a4}
\usepackage[utf8]{inputenc}
\usepackage[T1]{fontenc}
\usepackage{amsmath}
\usepackage{amsfonts,amssymb,amsthm}
\usepackage[color=green!30]{todonotes}
\usepackage{xcolor}
\usepackage{enumitem}
\usepackage[colorlinks,
    linkcolor={red!50!black},
    citecolor={blue!50!black},
    urlcolor={blue!80!black}]{hyperref}
\usepackage[noabbrev,nameinlink,capitalise]{cleveref}
\usepackage[xy]{qsymbols}
\usepackage{tikz-cd}
\usetikzlibrary{arrows}
\usepackage{mathtools}
\usepackage{array,arydshln}
\usepackage{bigdelim}
 \usetikzlibrary {arrows.meta} 
 \usepackage{mathrsfs}
 \usepackage{graphicx,scalerel}
 \newcommand\wh[1]{\hstretch{3}{\hat{\hstretch{0.33}{#1\kern+0.12em}}}}

\newcommand{\RR}{\mathbb R}
\newcommand{\CC}{\mathbb C}
\newcommand{\ZZ}{\mathbb Z}

\DeclareMathOperator{\Id}{Id}
\newcommand{\bma}{\begin{pmatrix}}
\newcommand{\ema}{\end{pmatrix}}
\newcommand{\bsm}{\left(\begin{smallmatrix}}
\newcommand{\esm}{\end{smallmatrix}\right)}
\newcommand{\dd}{\mathrm{d}}
\DeclareMathOperator{\vN}{vN}
\newcommand{\detFK}{{\det}_{\mathrm{FK}}}
\renewcommand{\Re}{\operatorname{Re}}

\DeclareMathOperator{\lk}{lk}
\DeclareMathOperator{\im}{im}
\DeclareMathOperator{\tr}{tr}
\newcommand{\bs}[1]{{\boldsymbol{#1}}}
\newcommand{\norm}[1]{\lVert{#1}\rVert}

\newtheorem{theorem}{Theorem}

\newtheorem{corollary}[theorem]{Corollary}
\newtheorem{lemma}[theorem]{Lemma}
\newtheorem{proposition}[theorem]{Proposition}

\theoremstyle{definition}
\newtheorem{definition}[theorem]{Definition}

\newtheorem{example}[theorem]{Example}

\newtheorem{remark}[theorem]{Remark}


\crefname{collitem}{}{}
\creflabelformat{collitem}{#2\cref{prop:collect}~(#1)#3}

\begin{document}
\title[Triangulations and combinatorial zeta functions]{Combinatorial zeta functions counting triangles}
\author{Leo Benard}
\address{Mathematisches Institut, Georg--August Universit\"at G\"ottingen $\&$
Institut de Mathématiques de Marseille, Aix--Marseille University}
\email{leo.benard@univ-amu.fr}
\author{Yann Chaubet}
\address{Department of Pure Mathematics and Mathematical Statistics, University of
Cambridge, Cambridge}
\email{y.chaubet@dpmms.cam.ac.uk}
\author{Nguyen Viet Dang}
\address{Sorbonne Université and Université Paris Cité, CNRS, IMJ-PRG, F-75005 Paris, France.\\
Institut Universitaire de France, Paris, France.}
\email{dang@imj-prg.fr }
\author{Thomas Schick}
\address{Mathematisches Institut, Georg--August Universit\"at  G\"ottingen}
\email{thomas.schick@math.uni-goettingen.de}

\begin{abstract}
In this paper, we compute special values of certain combinatorial zeta
functions counting geodesic paths in the $(n-1)$-skeleton of a triangulation
of an $n$-dimensional manifold. We show that they carry a topological
meaning. As such, we recover the first Betti and $L^2$-Betti numbers of
compact manifolds, and the linking number of pairs of null-homologous knots in
a 3-manifold.

The tool to relate the two sides (counting geodesics/topological invariants)
are random walks on higher dimensional skeleta of the triangulation.
\end{abstract}

\maketitle

\section{Introduction}
Given a hyperbolic surface $\Sigma$, Fried \cite{Fried} initiated the study of the behavior at the origin of the so-called Ruelle zeta function \cite{ruelle1976zeta}
$$
\zeta_\Sigma(s) = \prod_{\gamma} \left(1-e^{-s \ell(\gamma)}\right),
$$
where the product runs over all primitive closed geodesics of $\Sigma$, $\ell(\gamma)$ is the length of $\gamma$. This
infinite product is convergent whenever $\Re s > 1$, and admits a meromorphic
extension for $s$ lying in the whole complex plane. It
is closely related to the Selberg zeta function
$$
S_\Sigma(s)=\prod_{\gamma}\prod_{k=0}^\infty (1-e^{-(s+k)\ell(\gamma)})
$$
introduced and studied in \cite{Selberg}. Indeed,
$$
\zeta_\Sigma(s)=S_\Sigma(s)/S_{\Sigma}(s+1)
$$
and the meromorphic extension of
$\zeta_\Sigma(s)$ then follows from Selberg's work
\cite[p.~75]{Selberg}). We also mention the work~\cite{Fried1} which studies the relation between Ruelle and Selberg zeta functions and where the continuation is proved using dynamical methods. The zeros of $S_\Sigma(s)$ are 
related to the spectrum of the Laplacian. Indeed, one way of proving the analytic continuation of Selberg and Ruelle zeta functions (and also Poincar\'e series) is to relate these dynamical counting functions to the Laplacian and use the spectral theory of the Laplacian to prove meromorphic extension. 
This mechanism, which is sometimes called quantum--classical correspondence,
has a quite immediate interpretation in the combinatorial setting and it is the goal of the present paper to illustrate it by simple yet striking examples. 
In \cite{Fried}, Fried showed that $\zeta_\Sigma(s)$ vanishes at order $-\chi(\Sigma)$ at $s=0$, and this result was extended to arbitrary negatively curved surfaces by Dyatlov--Zworski \cite{DZ}. An important consequence is that the topology of a negatively curved surface is fully determined by its length spectrum, the set of lengths of primitive closed geodesics.

%

The aim of this note is to prove analogous statements in a combinatorial
setting. Throughout, $M$ is a compact connected oriented manifold of dimension
$n > 1$, endowed with a triangulation $\mathscr T$. In fact, as in the
classical case of the Ruelle and Selberg zeta functions, we prove a relation
between combinatorial closed geodesics entering the definition of a
combinatorial zeta function and the spectrum of the combinatorial Laplacian.

\begin{definition}\label{def:geodesic}
  A \textit{combinatorial geodesic path} in the $n-1$ skeleton
  $\mathscr T^{(n-1)}$ of the triangulation $ \mathscr T$ is a finite sequence
  $c = (\sigma_1, \dots, \sigma_q)$ of adjacent $(n-1)$-simplices such that no
  pair $(\sigma_k, \sigma_{k+1})$ of consecutive simplices bound the same
  $n$-simplex. A combinatorial geodesic path $c = (\sigma_1, \dots, \sigma_q)$ is
  \textit{closed} if $\sigma_q$ is adjacent to $\sigma_1$ and
  $(\sigma_q, \sigma_1)$ do not bound the same simplex. A
  \textit{combinatorial closed
    geodesic} is an equivalence class of closed geodesic paths,
  where we identify two closed geodesic paths if one is a cyclic permutation
  of the other. We denote by $\mathcal P$ the (in general infinite) set of
  primitive combinatorial closed geodesics, where a combinatorial closed
  geodesic is called primitive if it 
  is not a power of a shorter one. The length of a combinatorial closed geodesic
  $\gamma = [(\sigma_1, \dots, \sigma_q)]$ is denoted by $|\gamma| = q$ while
  $\varepsilon_\gamma \in \{-1, 1\}$ denotes its \textit{reversing index},
  which is the parity of how often orientations are flipped traversing the
  combinatorial closed geodesic and which is defined in Definition \ref{def:reversing_ind}.
\end{definition}

Note that, traversing a combinatorial closed geodesic backwards will give
another, in general different combinatorial closed geodesic. 

The first main result of this paper is the following theorem. 
\begin{theorem}\label{theo:mainfried}
Assume that $M$ is a compact oriented manifold of dimension $n\geqslant 2$ with
triangulation $\mathscr T$. The combinatorial zeta function 
$$
\zeta_\mathscr T(z) = \prod_{\gamma\in \mathcal{P}}\left(1-\varepsilon_\gamma z^{|\gamma|}\right),
$$
converges for $|z|$ small enough. It is a polynomial function of degree $|\mathscr T^{(n-1)}|$ in $z$ and vanishes of order $b_1(M)$ at $z=(n+2)^{-1}$.
Here, $|\mathscr T^{(n-1)}|$ denotes the cardinality of the $(n-1)-$skeleton
of the triangulation.
\end{theorem}

\begin{remark}
  Note that, by Poincar\'e duality, $b_1(M)=b_{n-1}(M)\leqslant |\mathscr
  T^{(n-1)}|$ so that the vanishing order in Theorem \ref{theo:mainfried}
  indeed is bounded by the degree of the polynomial.
\end{remark}

A consequence of \cref{theo:mainfried} is the following result.

\begin{corollary}\label{corol:1stBetti}
 The numbers $(n_k)_{k\in\mathbb{N}}$ (resp $(m_k)_{k\in \mathbb{N}}$) of all combinatorial closed geodesics of lengths $k \leqslant |\mathscr T^{(n-1)}|$ and reversing index $+1$ (resp $-1$) determine the first Betti number of a compact manifold $M$.
 \end{corollary}
 \begin{proof}
By \cref{theo:mainfried}, expanding formally the infinite product $\zeta_\mathscr T(z)$ in $z$ yields a priori a formal power
series in $z$. However, we know that $\zeta_\mathscr T(z)$ is in fact a polynomial of degree $|\mathscr T^{(n-1)}| $, so all the powers of $z$ of degree $>|\mathscr T^{(n-1)}| $ vanish in the power series expansion due to certain cancellations. Taking just the product for all $|\gamma|
\leqslant |\mathscr T^{(n-1)}|$:
$$P(z)=\prod_{\gamma\in \mathcal{P}, \vert\gamma\vert\leqslant |\mathscr T^{(n-1)}|}\left(1-\varepsilon_\gamma z^{|\gamma|}\right)$$
gives a polynomial (of high degree) $P$ which decomposes as:
$$P(z)=\zeta_\mathscr T(z) +z^{|\mathscr T^{(n-1)}|+1}\mathbb{Z}[z],$$ so $P$ exactly equals 
the full combinatorial zeta function modulo terms of degree $\geqslant |\mathscr T^{(n-1)}|+1$. The
vanishing order of this polynomial near $z=(n+2)^{-1}$ therefore also gives
the first Betti number.
\end{proof}
\noindent Another straightforward corollary is the following:
\begin{corollary}\label{corol:2_and_4_mf}
The Euler characteristic of a connected orientable surface is determined
by its \emph{combinatorial length spectrum}, given by the set of pairs
$(\varepsilon_\gamma, |\gamma|)$ for $\gamma \in \mathcal P$ with $|\gamma|
\leqslant |\mathscr T_1|$.

More generally, the combinatorial data of the triangulation (more
specifically, the number of {simplices and of} combinatorial closed geodesics of index $\pm 1$)
is sufficient 
to recover all the Betti numbers of a triangulated compact connected
orientable manifold of dimension smaller than 5.
\end{corollary}

For a non-compact normal covering $\widehat M$ of $M$ the random
walk on the vertices and higher dimensional simplices of a lifted
triangulation gives information about the spectral properties of the Laplacian
near zero, which in turn determines topological $L^2$-invariants of the
manifold $M$. Similarly, our result has a version which holds for non-compact normal
coverings, as we explain now. We refer to section \ref{subsec:L2} for the
definitions related to $L^2$-invariants.

We start with a triangulation $\mathscr T$ of an $n$-dimensional compact manifold $M$. 
Given a quotient $\pi$ of the fundamental group $\pi_1(M)$, it acts on the
corresponding cover $\wh{M}$ of $M$  with its lifted triangulation
$\wh{\mathscr T}$. This defines a chain complex $C_*^{(2)}(\mathscr T, \pi)$
of Hilbert spaces with unitary action of $\pi$ (specifically, Hilbert $\mathcal N(\pi)$-modules). As $\mathbb{C}$-vector spaces these are (typically infinite-dimensional) Hilbert spaces, but the action of $\pi$ endows them with the structure of a direct sum of finitely many copies of the $L^2$-completion $\ell^{(2)}(\pi)$ of $\mathbb{C}\pi$, on which the von Neumann algebra $\mathcal N(\pi)$ acts naturally. The von Neumann algebra is the algebra  of $\pi$-equivariant bounded operators from $\ell^{(2)}(\pi)$ to itself. 

In turn, we can define a combinatorial Laplacian $\Delta_k^{(2)}$ acting as a
bounded operator on $C_k^{(2)}(\mathscr T, \pi)$, and the von Neumann
dimension (defined by \cref{eq:vNdim}) of its kernel is the $k$-th
$L^2$-Betti number $b_k^{(2)}(M, \pi)$. Let us remark that when $\pi$ is finite then
this recovers the classical Betti numbers normalized by multiplication
with $\frac{1}{|\pi|}$. One should think of the $L^2$-Betti numbers as
correspondingly normalized Betti numbers of $\widehat M$ which make sense even
if $\pi$ is infinite. If $\pi =\pi_1(M)$,
then one usually drops the group from the notation and refers to
$b^{(2)}_k(M)$ as the $k$-th $L^2$-Betti number of $M$.

We fix a fundamental domain $\mathcal F$ in $\wh{\mathscr{T}}$ for the action
of the group $\pi$, and we denote by $\wh{\mathcal P}$ the set of primitive
combinatorial closed geodesics in $\wh{\mathscr{T}}$ which start in $\mathcal F$ and travel
in $\wh{\mathscr{T}}$ (and of course end in $\mathcal F$).

The second main result of this paper is the following theorem:
\begin{theorem}
\label{theo:mainL2}
The combinatorial $L^2$-zeta function 
$$\zeta^{(2)}_{\wh{\mathscr{T}}}(z) = \prod_{\gamma\in {\wh{P}}}\left(1-\varepsilon_\gamma z^{|\gamma|}\right)
$$
converges for $|z|$ small enough, and it extends as a real analytic function
on the disk of diameter $(0, \frac{1}{n+2})$. Moreover,
$$
\zeta^{(2)}_{\wh{\mathscr T}}\left(\frac{1}{n+2}-z\right) = z^{b_1^{(2)}(M,\pi)} f(z) 
$$
with a function $f$ which is continuous at $0$.  If $\Delta_{n-1}^{(2)}$ is of
determinant class, then
\begin{equation*}
  f(0) =(n+2)^{2b_1^{(2)}(M,\pi)-|\mathscr T^{(n-1)}|} \cdot \detFK(\Delta_{n-1}^{(2)}).
\end{equation*}
If $\Delta_{n-1}^{(2)}$ is not of determinant class, then $f(0)=0$ but $f$
converges to $0$ slower than any power of $z$ in the sense that for all $C>0$,
$\alpha>0$ there is $\varepsilon>0$ such that
\begin{equation*}
  f(z)> Cz^\alpha,\qquad 0 < z < \varepsilon.
\end{equation*}
In particular, the function $\zeta^{(2)}_{\wh{\mathscr T}}$ and hence the
combinatorial closed geodesics on $\widehat M$ 
determine $b_{n-1}^{(2)}(M, \pi)$.
\end{theorem}

Note that \cref{theo:mainL2} says that the
behavior of the function $\zeta^{(2)}_{\wh{\mathscr{T}}}$ near  $\frac 1 {n+2}$
recovers the first $L^2$-Betti number $b_{1}^{(2)}$ just as in the finite
case. Of course, specializing to $\pi=\{1\}$, \cref{theo:mainL2} also makes a
statement about $M$ and its zeta function $\zeta_{\mathscr T}(s)$. The
Fuglede--Kadison determinant $\det_{\text{FK}}$ in this case is just the product of the non-zero
eigenvalues of the matrix $\Delta_{n-1}$.
Still, \cref{theo:mainfried} in this case  is slightly stronger: thanks to the
finiteness, the zeta function is shown to be a polynomial there. This cannot
be expected and is not true in general for combinatorial $L^2$-zeta functions. Note also that $L^2$-Betti numbers are not
necessarily integers; they are non-negative real numbers. A question of Atiyah
\cite{Atiyah} asked
if they are always rational, which has been answered negatively by
now, the first example in \cite{GLSZ}.

Concerning the determinant class condition, the determinant conjecture
\cite[Conjecture 13.2]{Lueck} would imply that any of the operators
$\Delta^{(2)}_k$ defined above are of determinant class. This conjecture is
known to be true for a large class $\mathcal G$ of groups $\pi$ described in
\cite[Section 13.1.3]{Lueck} and \cite{Schick}, and then for \emph{any}
covering $\wh M$, with covering group such a $\pi$, of any compact manifold $M$ (or
even compact CW-complex $M$). In particular, the class $\mathcal G$
contains every lattice in a matrix Lie group (and more generally every residually
finite group), all amenable groups, and is closed under extensions with amenable
quotients, colimits along directed systems of inclusions, inverse limits,
passage to a subgroup and quotients by finite kernels.


A striking interpretation of \cref{theo:mainL2} is the following. Recall first
that a combinatorial closed geodesics in $\wh{\mathscr T}^{(n-1)}$ is a sequence of
$(n-1)$-dimensional simplices which starts and ends at the same simplex. Such
a sequence retracts onto a closed loop (in the sense of a continuous map
$\gamma\colon S^1\to \widehat M$) in the manifold $\wh M$, which projects to a
closed loop in $M$ whose base point is not well defined. The latter
canonically represents a conjugacy class in $\ker (\pi_1(M) \to \pi)$.
\begin{corollary}
Counting combinatorial closed geodesics of all lengths $k\in\mathbb{N}$ in $\mathscr
T^{(n-1)}$ which represent a conjugacy
class in $\ker (\pi_1(M) \to \pi)$ recovers the first $L^2$-Betti number $b_1^{(2)}(M, \pi)$.

In particular, counting homotopically trivial
combinatorial closed geodesics of all lengths $k\in\mathbb{N}$ on a closed
surface $\Sigma$ recovers $b_1^{(2)}(\Sigma) = -\chi(\Sigma)$.
\end{corollary}

As a final remark, let us notice that each $(n-1)$-dimensional simplex of our triangulation lies in
the boundary of exactly two $n$-dimensional simplices and each $n$-simplex has
precisely $(n+1)n/2$  codimension $2$-faces (no identifications). This is the necessary
condition for our crucial \cref{eq:deltan+2} to hold and then all the results
hold, compare \cref{rem:generalize_trian}. This is weaker than
asking $M$ to be a manifold, it is just asking $M$ to be non-singular in
codimension one. We also make use of Poincaré duality, if this is not
available, we have to replace the first by the $n-1$-th ($L^2$)-Betti number
and all the results continue to hold. In particular, we don't need and don't
use a smooth structure on $M$.
\medbreak

The last main result of this paper concerns the linking number of knots.
Recently, the third author and Rivi\`ere showed in \cite{Dang_Riviere} that the linking number $\mathrm{lk}(\kappa_1, \kappa_2)$ of two homologically trivial curves $\kappa_1, \kappa_2$ in the unit tangent bundle of a negatively curved surface can be recovered as the regularized value at $s=0$ of the Poincaré series
$$
\eta(s) = \sum_{c} e^{-s \ell(c)},
$$
where the sum runs over all geodesic paths $c$ joining orthogonally the projections of $\kappa_1$ and of $\kappa_2$ on the surface $\Sigma$ (see also \cite{chaubet2022poincare} for similar results on surfaces with boundary).

Our last result is a combinatorial {analog} of this result. We fix any compact
$3$-manifold  $M$, again with a triangulation $\mathscr T$.

We let $\mathscr T^\vee$ be the dual  {polyhedral} decomposition of $\mathscr
T$ (see \S\ref{sec:lapl} and  \cref{fig:dual}). Let $\kappa_1$ and $\kappa_2$
be two rationally null-homologous, oriented knots in $M$ which are one-chains
in $\mathscr T$ or $\mathscr T^\vee$, respectively. Note that every pair of null-homologous knots is homologous to such a pair $(\kappa_1, \kappa_2)$.
We denote by $\mathcal G^\perp(\kappa_1, \kappa_2)$ the set of
{\emph{orthogeodesic paths}} from $\kappa_1$ to $\kappa_2$, meaning by
definition combinatorial geodesic paths $c = (\tau_1, \ldots, \tau_q)$ in the $2$-skeleton of $\mathscr T$ such that the first simplex $\tau_1$ of $c$ bounds part of $\kappa_1$, and $\kappa_2$ intersects the last simplex $\tau_q$ of $c$, as in  \cref{fig:link}.

For $c \in \mathcal G^\perp(\kappa_1, \kappa_2)$, we denote by $\varepsilon_c
\in \{-1, 1\}$ its reversing index and by $m_c \in \ZZ$ its \textit{incidence
  number} on $(\kappa_1, \kappa_2)$; we refer to \S\ref{sec:link} for a proper
definition of these quantities --- we nevertheless mention that if the knots
$\kappa_j$ are \textit{simple}, in the sense that each $2$-simplex of
$\mathscr T$ is touched by $\kappa_1$ at most once and each $1$-cell of
$\mathscr T^\vee$ appears at most once in $\kappa_2$, then $m_c \in \{-1, 1\}.$

\begin{figure}[h]
\def\svgwidth{0.4\columnwidth}
\includegraphics[scale=0.5]{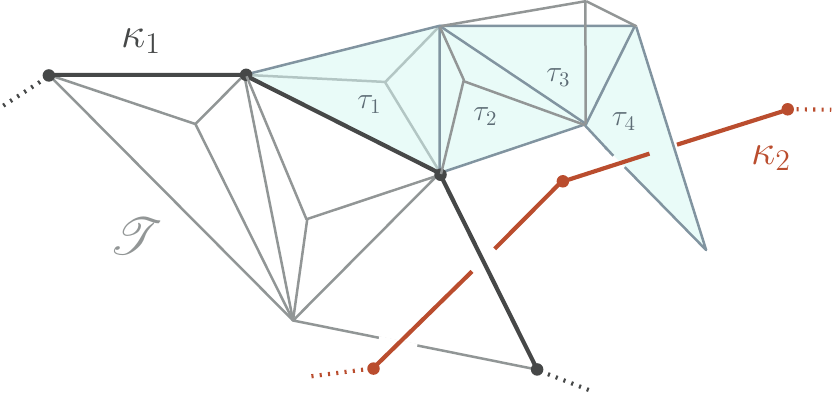}
\caption{
\label{fig:link} An orthogeodesic $c = (\tau_1, \tau_2, \tau_3, \tau_4)$ linking $\kappa_1$ and $\kappa_2$. The knot $\kappa_1$ is a one-chain in $\mathscr T$ while $\kappa_2$ is a one-chain in $\mathscr T^\vee$.
}
\end{figure}

Our final main theorem is the following result:
\begin{theorem}
\label{theo:mainlink}
The combinatorial Poincaré series
$$
\eta(z) = \sum_{c \in \mathcal G^\perp(\kappa_1, \kappa_2)} \varepsilon_c m_c z^{|c|}
$$ 
converges whenever $|z|$ is small enough. Moreover, it extends to a rational
function in $z$, regular at $z=1/(n+2)$, 
and
$$
\eta\left(\frac 1 {n+2}\right) = \lk(\kappa_1, \kappa_2).
$$
 
\end{theorem}


\medbreak

The proofs of all the results stated so far rely on the following simple but
fundamental fact (see \cref{prop:Lapl}):
the \emph{combinatorial Laplacian} introduced in \S\ref{sec:lapl} acts on the $(n-1)$-skeleton of any triangulation $\mathscr T$ as
\begin{equation}\label{eq:deltan+2}
\Delta = (n+2) \Id - T
\end{equation}
where $T$ is the transfer matrix of a signed geodesic random walk. {The above fact is
  a topological version of the famous Brydges--Fr\"ohlich--Sokal random walk
  representation widely used in Quantum Field
  Theory~\cite{fernandez2013random, GJ}}.

The random walk we consider is not a random walk on the vertex set of the
triangulation of $M$ or of the cover $\wh{M}$, but on the set of $n-1$
simplices. The former has been used famously to obtain information about
another topological $L^2$-invariant, namely the lowest Novikov-Shubin
invariant. Here, fundamental work of Varopoulos \cite{Varopoulos} implies that
the lowest Novikov-Shubin invariant is encoded in the sequence $p(n)$, where
$p(n)$ is the probability to return to the starting point after 
$n$ steps. Note that this is closely related to the number of combinatorial closed geodesics of
length $n$ in the $1$-skeleton.

In the G\"ottingen doctoral thesis \cite{Hoepfner} of Tim H\"opfner, a
generalization of the result of Varopoulos to Novikov-Shubin invariants of
higher degree $k$ is obtained. H\"opfner shows that they can be obtained from a
suitable signed random walk on the $k$-cells of the universal cover $\widetilde M$ of $M$. Our article adds to this 
circle of ideas by expressing the first (by Poincar\'e duality equal to $n-1$)
$L^2$-Betti number via a random walk on the $(n-1)$ skeleton.

In our proofs, we do crucially use the transfer matrix $T$ and its powers and
analyze its combinatorial and analytic features. This way, we analyze and use
the signed random walk described by $T$. We do not explicitly refer to
probabilistic results on this random walk, but we believe that our
investigation sheds light on the relation between this random walk and fine
topological, geometric and spectral properties of the manifold on which the random walk
takes place.

\medbreak

In the same flavour as the present work, we would like to mention two results that have been communicated to us:
\begin{itemize}
\item by Dang--Mehmeti \cite{DM24}, for $\Gamma$ a Schottky group acting on the Berkovich projective line, 
they are able to recover the number $g$ of generators of $\Gamma$ from the value at $s=0$ of a similar Poincar\'e series,
\item unpublished work by Anantharaman~\cite{Anantharaman} who is able to recover the Euler characteristic of a metric graph from the behaviour at $s=0$ of similar zeta functions
and Poincaré series as in the present paper.
\end{itemize}


\subsection*{Organization of the paper}
Preliminary background on combinatorial Laplacians and $L^2$-invariants is gathered in \S\ref{sec:prelim}, where in particular \eqref{eq:deltan+2} is proved.
Then in \S\ref{sec:fried} we prove \cref{theo:mainfried} with \cref{corol:2_and_4_mf} and \cref{theo:mainL2}, and in \S\ref{sec:link} we prove  \cref{theo:mainlink}.

\subsection*{Acknowledgments}
We thank Jean Raimbault for several interesting comments on a preliminary version of our results. N.V.D would like to thank Jean Yves Welschinger for some interesting discussion on our results.
L.B.~and T.S.~were partially funded by the Research Training Group 2491
``Fourier Analysis and Spectral Theory'', University of Göttingen. Y. C. is
supported by the Herchel Smith Postdoctoral Fellowship Fund. N.V.D is
supported by the Institut Universitaire de France.

We thank three referees for carefully reading the first version of the paper
and making many helpful remarks which significantly improved the presentation
of the paper.

\section{Preliminaries}
\label{sec:prelim}
In \S\ref{sec:lapl} we develop the combinatorial setting that will be used throughout the paper, and in \S\ref{subsec:L2} we collect the necessary notions on $L^2$-invariants.
\subsection{Combinatorial Laplacian}
\label{sec:lapl}
We start with a given triangulation $\mathscr T$ of an $n$-dimensional
manifold $M$. For us, part of the data of the triangulation is an orientation
of all its simplices.

We fix the convention that the orientation of a simplex $\sigma$ provides an
induced orientation for each simplex of its boundary by the \emph{first vector
  pointing outward}. More precisely: at any point $p$ of $\tau \in \partial
\sigma$, a basis $ b$ of $T_p\tau$ is positive if and only if the basis
$n_p \oplus  b$ yield a positive basis of $T_p\sigma$, where $n_p$ is a
normal vector to $\tau$ in $\sigma$ pointing outward.

More combinatorially, an orientation of a simplex is given by the class of an ordering of
its vertices, two such orderings being equivalent if one can be obtained from
the other by an even permutation. The negative orientation of a given one is
represented by the other class of orderings of the vertices. The orientation
induced on the face $\tau$ of
an oriented simplex $\sigma$ obtained by leaving out vertex number $k+1$ is
defined to be  $(-1)^k$ times the orientation
represented by the restriction of the ordering of the simplices of $\sigma$ to $\tau$.

The boundary of a simplex $\sigma$ is given by
\begin{equation}
\label{eq:boundary}
\partial \sigma = \sum_{\tau \text{ hyperface of } \sigma} \varepsilon(\tau) \tau, 
\end{equation}
where $\varepsilon(\tau) = \pm 1$ according to whether the orientation of $\tau$ coincides or not with the orientation induced by $\sigma$.

We denote by $C_k(\mathscr T)_\ZZ$ the $\ZZ$-module generated by the
$k$-dimensional simplices of $\mathscr T$. Complexifying, this yields the
simplicial chain complex 
$$
0 \overset{\partial}{\longrightarrow} C_n(\mathscr T) \overset{\partial}{\longrightarrow} \cdots \overset{\partial}{\longrightarrow} C_0(\mathscr T) \overset{\partial}{\longrightarrow} 0
$$
whose homology is of course the homology of the manifold $M$ (with complex
coefficients).

\begin{definition}
\label{def:compa}
For $k>0$, a pair of $k$-simplices is said to be \textit{admissible} if they share a
common $(k-1)$-face, and do not bound  {the} same $(k+1)$-simplex. Two
$k$-simplices forming an admissible pair have a \emph{compatible orientation}
if they induce opposite orientations on their common $(k-1)$-hyperface.
\end{definition}

The dual complex $\mathscr T^\vee$ of the triangulation $\mathscr T$ is the cell complex constructed as follows (see \cite[Chapter VI, Section 6]{Bredon} or \cite{STbook}). Take  $\mathscr T'$ the barycentric subdivision of $\mathscr T$. The closed star in  $\mathscr T^\vee$ of any vertex $p$ of $\mathscr T^{(0)}$ is an $n$-cell in $\mathscr T'$, denoted by $p^\vee$. If non-empty, i.e.~if $p$ and $q$ are adjacent, the intersection of $p^\vee$ and $q^\vee$ is an $(n-1)$-cell $\langle p,q\rangle^\vee$, which is dual to the edge $\langle p,q\rangle$ of $\mathscr T^{(1)}$, in the sense that it intersects this edge once positively, and does not intersect any other. Similarly, the intersection of 3 $n$-cells $p^\vee, q^\vee$ and $r^\vee$, if not empty, is a $(n-2)$-cell $\langle p,q,r \rangle^\vee$, dual to the 2-simplex $\langle p,q,r \rangle$.
This defines a Hodge star map
$$
\star \colon C_i(\mathscr T) \to C_{n-i}(\mathscr T^\vee),
$$
where the orientation of $\star \sigma$ is chosen so that at the intersection of $\sigma$ with $\star \sigma$, 
the orientation of $M$ induced by $\sigma$ followed by $\star \sigma$ is positive.
\begin{figure}[h]
\caption{
\label{fig:dual}
A $2$-dimensional triangulation $\mathscr T$ (in black), together with a dual {polyhedral} decomposition $\mathscr T^\vee$ (in red).
}
\vspace{0.2cm}
\includegraphics[scale=0.45]{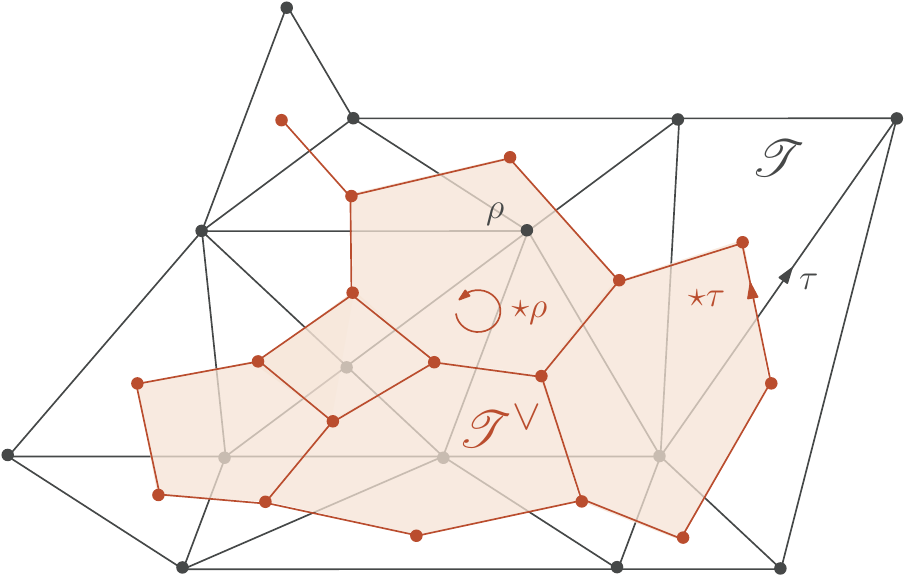}
\end{figure}
Since $\mathscr T$ is the dual complex of $\mathscr T^\vee$, we have a map $\star : C_\bullet(\mathscr T^\vee) \to C_\bullet(\mathscr T)$, and our choice of orientation yields
\begin{equation}\label{eq:involution}
\star^2 = (-1)^{k(n-k)} \mathrm{Id} \quad \text{on} \quad C_k(\mathscr T)
\end{equation}
for each $k = 0, \dots, n$.

Note that the dual complex is not a simplicial complex, in general, and not
even a particularly nice polyhedral complex. Fortunately, this is of no
relevance for our considerations where the specific computations are all
carried out on $\mathscr T$.

The family of oriented $k$-dimensional simplices of $\mathscr T$ defines a
basis $\beta_k$ of the chain complex $C_k(\mathscr T)$ for all $k$. We define a scalar product $\langle \cdot, \cdot \rangle$ on $C_k(\mathscr T)$ by declaring that this family is an orthonormal basis.
We let $\partial^\star : C_\bullet(\mathscr T) \to C_{\bullet + 1}(\mathscr T)$ be the adjoint operator of $\partial$ with respect to $\langle \cdot, \cdot \rangle$, which is defined by
$$
\langle \partial \sigma, \tau \rangle = \langle \sigma, \partial^\star \tau \rangle
$$
for every $\sigma \in C_k(\mathscr T)$ and $\tau \in C_{k-1}(\mathscr T)$. 

Note that there is an intersection product 
$$\cap \colon  C_k(\mathscr T) \times C_{n-k}(\mathscr T^\vee)$$
obtained by extending linearly the relation $\sigma \cap \tau^\vee = 1$ if $\tau=\sigma$ and $\sigma \cap \tau^\vee =0$ otherwise, for $\sigma, \tau$ two $k$-simplices of the basis $ \beta_k$. Clearly $\sigma\cap \tau^\vee = \langle \sigma, \tau\rangle$.
We obtain
\begin{lemma}
\begin{equation}\label{eq:partialstar}
\partial^\star = (-1)^{n(k + 1)} \star \partial \star \quad \text{on} \quad C_k(\mathscr T).
\end{equation}
\end{lemma}

\begin{proof}
Let $\tau \in C_k(\mathscr T)$. Then for all $\sigma \in C_{k+1}(\mathscr T)$, one has
$$\langle \sigma, \partial^\star \tau\rangle = \langle \partial \sigma, \tau\rangle = \partial \sigma \cap \star \tau.$$
Now we use the relations (\cite[Chapter 10, \S 69, equations (2) and (5)]{STbook})
$$\tau_1 \cap \tau_2 = (-1)^{k(n-k)} \tau_2 \cap \tau_1 \quad \text{and} \quad \tau_1 \cap \partial \tau_3 = (-1)^k \partial \tau_1 \cap \tau_3$$
for $\tau_1 \in C_k(\mathscr T), \tau_2 \in C_{n-k}(\mathscr T^\vee)$ and $\tau_3 \in C_{n-k+1}(\mathscr T^\vee)$. We obtain 
$$
\langle \sigma, \partial^\star \tau\rangle =\partial \sigma \cap \star \tau= (-1)^{k(n-k)} (\star \tau) \cap \partial \sigma = (-1)^{n-k}(-1)^{k(n-k)} \partial (\star \tau) \cap \sigma,
$$
which reads
$$
\begin{aligned}
&(-1)^{n-k}(-1)^{k(n-k)} (-1)^{(n-k-1)(k+1)} \sigma \cap \partial (\star\tau)\\ 
&\qquad \qquad = (-1)^{n-k}(-1)^{k(n-k)} (-1)^{(n-k-1)(k+1)}(-1)^{(n-k-1)(k+1)} \sigma \cap \star^2 \partial (\star\tau)\\
&\qquad \qquad=(-1)^{(k+1)(n-k)} \sigma \cap \star (\star \partial \star \tau) \\
&\qquad \qquad=(-1)^{n(k+1)} \langle \sigma, \star \partial \star \tau \rangle
\end{aligned}
$$
and the result is proved.
\end{proof}

\begin{definition}
The \emph{combinatorial Laplacian} is the operator
$$\Delta = \partial \partial^\star + \partial^\star \partial \colon C_\bullet(\mathscr T) \to C_\bullet(\mathscr T).$$
\end{definition}

The results we will prove in this paper all rely on the following proposition:
\begin{proposition}
\label{prop:Lapl}
Assume that $n>1$. The combinatorial Laplacian satisfies
\begin{equation}
  \label{eq:transfer_eq}
  \Delta_{n-1} = (n+2) \Id - T
\end{equation}
where $T$ 
is the
\emph{transfer signed random walk operator} defined in terms of the basis
$ \beta_{n-1}$ by
\begin{equation*}
  T(\tau) = \sum_{\substack{\sigma\in \bs\beta_{n-1} \\(\sigma,\tau)\text{ admissible}}}
  \varepsilon_{\tau, \sigma} \sigma,\qquad\tau\in  \beta_{n-1},
\end{equation*}
with $\varepsilon_{\tau,\sigma}=1$ if $\sigma$ and $\tau$ have compatible
orientations and $\varepsilon_{\tau,\sigma}=-1$ otherwise.
\end{proposition}

\begin{proof}
Given $\sigma_1, \sigma_2$ two simplices in the basis  $ \beta_{n-1}$, we compute 
$$\langle \Delta \sigma_1, \sigma_2\rangle = \langle \partial^\star \sigma_1, \partial^\star \sigma_2\rangle + \langle \partial \sigma_1, \partial \sigma_2\rangle.$$
Note that for $j = 1,2$, there are exactly two $n$-dimensional simplices
$\tau^+_j, \tau^-_j$ containing $\sigma_j$ in their boundary and signs
$\varepsilon_j^+,\varepsilon_j^-\in \{-1,1\}$ such that 
$$
\partial^\star \sigma_j = \varepsilon_j^+\tau^+_j + \varepsilon_j^-\tau^-_j.
$$
We distinguish four cases:
\begin{enumerate}
\item If $\sigma_1 = \sigma_2$, then $\tau_1^\pm = \tau_2^\pm$, $\varepsilon_1^{\pm}=\varepsilon_2^\pm$ and thus
  \begin{equation*}
    \langle \partial^\star \sigma_1, \partial^\star \sigma_2\rangle = \langle
    \varepsilon_1^+\tau^+_1 + \varepsilon_1^- \tau^-_1, \varepsilon_1^+\tau^+_1 + \varepsilon_1^-\tau^-_1 \rangle = 2.
  \end{equation*}
  On the other hand, $\partial \sigma_1$ consists of $n$ simplices of dimension $(n-2)$, hence 
$\langle \partial \sigma_1, \partial \sigma_2\rangle = n$ and we conclude
$$
\langle \Delta \sigma_1, \sigma_2 \rangle = n+2.
$$
\item
If $\sigma_1 \neq \sigma_2$ lie in the boundary of a common $n$-dimensional
simplex~$\tau$, up to replacing $\tau$ by $-\tau$ and exchanging the roles of $\tau_j^+$ and $\tau_j^-$, one can assume 
$$\tau = \tau_1^+ = \varepsilon \tau_2^+,$$
for some $\varepsilon \in \pm1$. Equivalently, the simplicial boundary of
$\tau$ contains $\sigma_1$ with sign $1$ and $\sigma_2$ with sign $\varepsilon$,
\begin{equation*}
  \langle \partial \tau,\sigma_1\rangle = 1;\quad \langle
  \partial\tau,\sigma_2\rangle = \varepsilon.
\end{equation*}
In particular we get
$$\langle \partial^\star \sigma_1, \partial^\star\sigma_2\rangle = \varepsilon.$$
As $n>1$ and we deal with a triangulation in this situation there is exactly one $(n-2)$-dimensional
simplex $\nu$ appearing in both $\partial \sigma_1$ and $\partial
\sigma_2$. We now claim that
$$
\langle \partial \sigma_1, \partial \sigma_2\rangle = -\varepsilon.
$$
To prove this, consider the simplicial chain
$0=\partial^2 \tau$. In the basis $\beta_{n-1}$ we have
\begin{equation*}
  \partial\tau = \sigma_1 + \varepsilon\sigma_2 + \sum_j \varepsilon_j\sigma_j
\end{equation*}
where the $\sigma_j$ are the faces of $\tau$ different of
$\sigma_1,\sigma_2$. Then $\langle\partial \sigma_j,\nu\rangle=0$ for those
$\sigma_j$ and
\begin{equation*}
 0 =\langle\partial^2\tau,\nu\rangle= \langle\partial\sigma_1+\varepsilon\partial\sigma_2,\nu\rangle
\end{equation*}
As $\nu$ is the only common term in the simplicial basis expansion of
$\partial\sigma_1$ and $\partial\sigma_2$ this implies indeed
$\langle\partial\sigma_1,\partial\sigma_2\rangle =-\varepsilon$.

Combining these computations we get
\begin{equation*}
  \langle \Delta \sigma_1, \sigma_2 \rangle
  =\langle\partial^*\sigma_1,\partial^*\sigma_2\rangle + \langle\partial\sigma_1,\partial\sigma_2\rangle = 0.
\end{equation*}
\item
If $\sigma_1$ and $\sigma_2$ share a common $(n-2)$-face, but are not in the boundary of a common $n$-dimensional simplex, then 
$\langle \partial^\star \sigma_1, \partial^\star \sigma_2 \rangle = 0$, and
$\langle \partial \sigma_1, \partial \sigma_2\rangle = \pm 1$. Moreover, if
$\sigma_1$ and $\sigma_2$ have compatible orientations, then by definition of
this notion $\langle \partial \sigma_1, \partial \sigma_2\rangle = -1$ and $\langle \partial \sigma_1, \partial \sigma_2\rangle =+1$ otherwise.
\item 
Finally, if $\sigma_1$ and $\sigma_2$ do not share a common face, then both $\langle \partial^\star \sigma_1, \partial^\star \sigma_2 \rangle$ and $\langle \partial \sigma_1, \partial \sigma_2 \rangle$ vanish.
\end{enumerate}
The proof is complete.
\end{proof}

\begin{remark}\label{rem:generalize_trian}
 \cref{prop:Lapl} is the only property of the cell decomposition and resulting
 cellular Laplacian which is used in the following (together with Poincar\'e
 duality to identify the first and the codimension $1$ Betti numbers).

 We observe that the proof
 of \cref{prop:Lapl}
 works with slightly weaker conditions on the triangulation of our
 $n$-dimensional space: it suffices to
 have a decomposition into simplices such that each codimension $1$ simplex is
 in the boundary of exactly two $n$-dimensional simplices and such that each
 top dimensional simplex has exactly $(n+1)n/2$ distinct faces of codimension
 $2$, so that two codimension $1$ simplices intersect in at most one simplex
 of codimension $2$.

 Therefore, we can work with polyhedral decompositions into simplices which
 are not triangulations, as long as they satisfy these conditions.

 An example is illustrated in \cref{fig:torus}.
\end{remark}
\subsection{$L^2$-invariants}
\label{subsec:L2}
In this subsection, we consider the situation where the triangulation
defined in the previous subsection admits a free co-compact simplicial action of a group $\pi$. In this setting, we will denote by $\wh{\mathscr{T}}$ the triangulation, and by $\mathscr T = \wh{\mathscr{T}} / \pi$ the quotient, which we assume to be a triangulation of a compact manifold $M$ of dimension $n$. In other words, $\pi$ is a quotient of $\pi_1(M)$ and $\wh{\mathscr{T}}$ is a triangulation of the (in general non-compact) corresponding covering $\wh{M}$ of $M$.

The action of $\pi$ on $\wh{\mathscr{T}}$ gives the groups
$C_k(\wh{\mathscr{T}})_\ZZ$ the structure of free $\ZZ[\pi]$-modules of rank
$|\mathscr T^{(k)}| = \dim C_k(\mathscr T)$. The combinatorial Laplacian
$\Delta_k$ described in \cref{sec:lapl} acts on $C_k(\wh{\mathscr{T}})$ as a
$\pi$-equivariant operator, in the sense that
$$
\Delta_k (\alpha \cdot \hat \sigma) = \alpha \cdot \Delta_k \hat \sigma, \quad
\forall \hat \sigma \in C_k(\wh{\mathscr T}), \quad \forall\alpha \in \pi,
$$
and so does the transfer operator $T$ defined in \cref{prop:Lapl}. 
Note that any choice of lifts $\{\widehat{\sigma}_i^k\}$ of a finite basis $\{\sigma_i^k\}$ of $C_k({\mathscr T})$ yields a finite basis of $C_k(\wh{\mathscr T})_\ZZ$ as a free $\ZZ[\pi]$-module.

Complexifying, we obtain $C_k(\wh{\mathscr T})$ as a free $\mathbb C[\pi]$-module, yielding a complex vector space, generated by elements of the form
$$
\alpha \cdot \hat \sigma_i^k, \quad \alpha \in \pi, \quad i = 1, \dots, |\mathscr T^{(k)}|.
$$
We endow this space with a scalar product $\langle \cdot, \cdot \rangle_{L^2}$, defined by
\begin{equation}\label{eq:scalprod}
\left\langle \sum_{\alpha \in \pi} \lambda_\alpha \, \alpha\cdot\hat{\sigma}_i^k, \, \sum_{\beta\in \pi} \mu_{\beta} \, \beta \cdot \hat \sigma_j^k \right\rangle_{L^2} =  \delta_{i,j}\sum_{\alpha \in \pi} \lambda_\alpha \overline\mu_{\alpha^{-1}}.
\end{equation}
We will denote by $C_k^{(2)}(\mathscr T, \pi)$ the Hilbert space given by the completion of $C_k(\wh{\mathscr T})$ with respect to the norm induced by $\langle \cdot, \cdot \rangle_{L^2}$.

Given a bounded $\pi$-equivariant operator $P \colon C_k^{(2)}(\mathscr T, \pi) \to C_k^{(2)}(\mathscr T, \pi)$, we define its \emph{von Neumann trace} $\tr_{\mathrm{vN}}P$ by 
\begin{equation}
\label{eq:L2tracedef}
\tr_{\mathrm{vN}}P = \sum_{i=1}^{|\mathscr T^{(k)}|} \langle P\hat{\sigma}_i^k, \hat{\sigma}_i^k\rangle_{L^2}.
\end{equation}
Given any $\pi$-invariant closed subspace $U \subset C_k^{(2)}(\mathscr T, \pi)$, we define its \emph{von Neumann dimension} by
\begin{equation}
\label{eq:vNdim}
\dim_{\mathrm{vN}}U  = \tr_{\mathrm{vN}}\Pi_U,
\end{equation}
where $\Pi_U$ is the orthogonal projection on $U$.

The operator $\Delta_k$  induces a bounded, self-adjoint $\pi$-equivariant operator
\begin{equation*}
  \Delta^{(2)}_k \colon C_k^{(2)}(\mathscr T, \pi) \to C_k^{(2)}(\mathscr T, \pi).
\end{equation*}
It has real, non-negative,  bounded spectrum. The spectral theorem for bounded self-adjoint operators yields a family of orthogonal projections
$$
\left(E_k(\lambda)\right)_{\lambda \in \RR}; \quad E_k(\lambda)=\chi_{[0,\lambda]}(\Delta^{(2)}_k)
$$
 called the spectral family of $\Delta^{(2)}_k$ (\cite[Definition
 1.68]{Lueck}); it has the property that \(\displaystyle \Delta_k^{(2)} = \int_\RR \lambda \, dE_k(\lambda)\). Then the $L^2$-spectral density function of $\Delta_k^{(2)}$ is the function
\begin{equation}
\label{eq:SDF}
D_k \colon \RR \to \RR_{\geqslant 0}, \quad  \lambda \mapsto \tr_{\mathrm{vN}}  E_k(\lambda). 
 \end{equation}
By definition, the \emph{$k$-th $L^2$-Betti number $b_k^{(2)}(M, \pi)$} is the von Neumann dimension of $\ker \Delta^{(2)}_k$, so that $b_k^{(2)}(M, \pi) = D_k(0)$.

Because $\Delta^{(2)}_k$ is bounded, for all $\lambda \geqslant \lVert \Delta^{(2)}_k\rVert$ the function $D_k$ is constant, equal to $D_k(\lambda)=|\mathscr T^{(k)}|$.

\begin{definition}
\label{def:detclass}
The Laplacian $\Delta_k^{(2)}$ is said to be \emph{of determinant class} if the Stieltjes integral
\begin{equation}
\label{eq:FKdef}
\int_{0^+}^\infty \log \lambda \, \dd D_k(\lambda) =
\lim_{\substack{\varepsilon \to 0\\ \varepsilon >0 }}
\int_{\varepsilon}^\infty \log \lambda \, \dd D_k(\lambda) \in \mathbb{R}
\end{equation}
exists as a real number, and in this case the \emph{Fuglede--Kadison determinant} of $\Delta_k^{(2)}$ is given by
$$
{\det}_{\mathrm{FK}}(\Delta_k^{(2)}) = \exp\int_{0^+}^\infty \log \lambda \,
\dd D_k(\lambda) \in (0,\infty).
$$
\end{definition}


\begin{remark}
  \strut
  \begin{itemize}
  \item The definition of Fuglede--Kadison determinant generalizes of course
    in the same way to any non-negative Hilbert $\vN\pi$-module endomorphism
    $A\colon U\to U$ (with $\dim_{\vN}(U)<\infty$), see \cite[Chapter 3]{Lueck}.
\item
The integrals on the right-hand side of \cref{eq:FKdef} are finite: the
boundedness of $\Delta^{(2)}_k$ implies that they are indeed integrals on
$[\varepsilon, \lVert \Delta^{(2)}_k\rVert]$.
\item By definition, if $\log(A)$ is defined by spectral calculus, i.e.~if
  $\sigma(A)\subset (0,\infty),$ then
  \begin{equation*}
    \log\detFK(A) = \tr_{\vN}(\log(A)) \quad\implies \quad\detFK(A)= \exp(\tr_{\vN}(\log(A))).
  \end{equation*}
\item In case the group $\pi$ is finite, i.e.~if
  we deal with a finite cover $\widehat M$ of $M$, then the FK-determinant
  is nothing but the 
  usual determinant of the positive matrix $\Delta^{(2)}_k$ restricted to the
  orthogonal of its kernel (i.e.~the product of the non-zero eigenvalues), but
  then raised to the power $1/|\pi|$.
\end{itemize}
\end{remark}

Later, we will use the following basic properties of the Fuglede--Kadison
determinant.
By \cite[Lemma 3.15
    (7)]{Lueck} we have 
  \begin{equation}\label{item:add}
    {\det}_{\mathrm{FK}}(A\oplus B) =
    {\det}_{\mathrm{FK}}(A)\cdot {\det}_{\mathrm{FK}}(B) .
  \end{equation}
 If $\lambda\in (0,\infty)$
    and $A\colon U\to U$ is an injective endomorphism of the Hilbert
    $\mathcal{N}\left(\pi\right)$-module $U$ of finite von Neumann dimension then
    \begin{equation}\label{item:mult} 
      {\det}_{\mathrm{FK}}(\lambda A) =
      \lambda^{\dim_{\vN}(U)}{\det}_{\mathrm{FK}}(A)
    \end{equation}
    by \cite[Theorem 3.14 (1) and (6)]{Lueck}.

    We also need the following continuity result for the Fuglede--Kadison
determinant, which is not 
explicitly treated in \cite[Chapter 3]{Lueck}.
\begin{lemma}\label{lem:FK}
  Let $A\colon U\to U$ be a positive injective Hilbert $\vN\pi$-module endomorphism where
  $\dim_{\vN}(U)<\infty$.
  Then $f(z)=z\mapsto \detFK(A+z\Id)$ is an analytic function defined on
  $(0,\infty)$. It extends continuously at $0$ with
  \begin{equation}\label{item:conv}
    \lim_{z\to 0^+} f(z) = \detFK(A).
  \end{equation}
  If $A$ is not of determinant class and hence $\detFK(A)=0$, then $f(z)$
  converges to $0$ for $z\to 0$ slower than any positive power of $z$. More
  precisely, for
  every $C>0$ and $\alpha>0$, there exists $\varepsilon>0$ such that
  \begin{equation*}
    f(z) > C z^\alpha,\qquad  z\in(0,\varepsilon).
  \end{equation*}
\end{lemma}
\begin{proof}
  Let $F(t)=\tr_{\vN}(\chi_{[0,t]}(A))$ be the spectral density function of
  $A$ (using measurable functional calculus with the characteristic function
  $\chi_{[0,t]}$ of the interval $[0,t]\subset\RR$).

  Then by the proof of \cite[Lemma 3.15 (5)]{Lueck} we have
  \begin{equation*}
    \detFK(A+z\Id) = \int_{0^+}^{\norm{A}} \log(\lambda+z)\,dF(\lambda).
\end{equation*}
Around each $z_0>0$, the function $z\mapsto \log(\lambda+z)$ has an absolutely
convergent power series expansion, with convergence uniform in $\lambda\in
[0,\norm{A}]$. Therefore,  
$\detFK(A+z\Id)$ is also analytic on $(0,\infty)$ by the continuity of the Stieltjes integral.

The continuity at $0$ is established in \cite[Lemma 3.15 (5)]{Lueck} with
$\lim_{\varepsilon\to 0^+}    {\det}_{\mathrm{FK}}(A+\varepsilon\Id)
={\det}_{\mathrm{FK}}(A)$.
  
Finally, to control the behaviour of $f$ near $0$, choose $C>0$ and
$\alpha>0$. On $(0,C^{-1/\alpha})$ we have $C\leqslant z^{-\alpha}$ and hence
$z^{\alpha}\geqslant C z^{2\alpha}$. Therefore, it suffices to find $\varepsilon>0$
such that $f(z)>z^\alpha$ on $(0,\varepsilon)$ to conclude the proof.

Now, we use that the spectral density function $F$ is right continuous and
increasing. Set ${\nu}=\frac{\alpha}{2\dim_{\vN}(U)}$. As $F(0)=0$ by the
injectivity of $A$, we can choose 
$\varepsilon>0$ such that $F(z^{\nu})<\frac{\alpha}{2}$ whenever $z<\varepsilon$.

With this $\varepsilon$ which we choose smaller than $1$ and for $0<z<\varepsilon<1$ we then have
\begin{equation*}
  \begin{split}
    \int_{0}^{\norm{A}} \log(\lambda+z)\,dF(\lambda) 
& = \int_0^{z^{\nu}}\log(\lambda+z)\,dF(\lambda)+\int_{z^{\nu}}^{\norm{A}}
  \log(\lambda+z)\,dF(\lambda)\\
  & \geqslant \int_0^{z^{\nu}}\log(z)\,dF(\lambda) +\int_{z^{\nu}}^{\norm{A}}
    \log(z^{\nu}+z)\,dF(\lambda)\\
&\geqslant F(z^{\nu})\log(z) + (\underbrace{F(\norm{A})}_{=\dim_{\vN}(U)}-F(z^{\nu})) \log(z^{\nu})\\
    & {\geqslant } \frac{\alpha}{2} \log(z) +
      \dim_{\vN}(U){\nu}\log(z)\\
    &{=} \alpha \log(z),
  \end{split}
\end{equation*}
where we used $\log(z)<0$ for $z < 1$ in the last inequality and $\nu = \alpha / (2 \dim_{\mathrm{vN}}(U))$ in the last equality.
Consequently, for $0<z<\varepsilon$ as chosen above,
\begin{equation*}
  \begin{split}
    \detFK(A+z\Id) & = \exp\left(\int_0^{\norm{A}} \log(\lambda+z)\,dF(\lambda)
                     \right)\\
    & \geqslant \exp(\alpha\log(z)) = z^\alpha.
  \end{split}
\end{equation*}
This completes the proof.
\end{proof}

\section{First Betti number and combinatorial Ruelle zeta function}
\label{sec:fried}
\subsection{The compact case}
The purpose of this paragraph is to prove \cref{theo:mainfried}.

\medbreak

Let us fix a compact oriented manifold $M$ of dimension $n>1$ with a triangulation~$\mathscr T$. 

\begin{definition}\label{def:reversing_ind}
For a combinatorial closed geodesic $\gamma$, we denote by $n_\gamma$ the
\textit{reversing number} of $\gamma$, that is, the number of adjacent pairs
with non-compatible orientation (as defined in Definition \ref{def:compa})
appearing in $\gamma$. Its parity 
$$
\varepsilon_\gamma = (-1)^{n_\gamma}
$$
is called the \textit{reversing index} of $\gamma$.
\end{definition}

For an illustration of this concept, see \cref{fig:sign}.

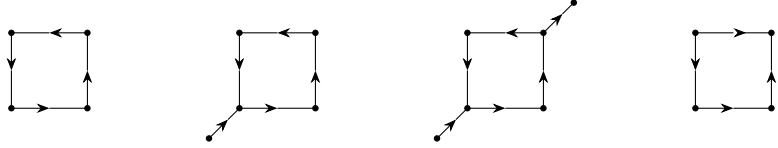
\begin{figure}[h]
\caption{\label{fig:sign} Some closed primitive unmarked geodesics $\gamma_i, \, i=1 \ldots 4,$ from left to right. The extra edges are browsed twice (back-and-forth). For the first one on the left, $n_{\gamma_1} = 0$, then $n_{\gamma_2}=1$ and $n_{\gamma_3} = n_{\gamma_4} = 2$.}
\begin{tikzpicture} 
\draw[fill=black] (-3,0) circle (1pt);
\draw[fill=black] (-2,0) circle (1pt);
\draw[fill=black] (-2,1) circle (1pt);
\draw[fill=black] (-3,1) circle (1pt);
\draw[-Stealth] (-3,0) -- (-2.5,0);
\draw[-Stealth] (-2,0) -- (-2,0.5);
\draw[-Stealth] (-2,1) -- (-2.5,1);
\draw[-Stealth] (-3,1) -- (-3,0.5);
\draw (-2.5,0) -- (-2,0);
\draw (-2,0.5) -- (-2,1);
\draw (-2.5,1) -- (-3,1);
\draw (-3,0.5) -- (-3,0);

\draw[fill=black] (0,0) circle (1pt);
\draw[fill=black] (1,0) circle (1pt);
\draw[fill=black] (1,1) circle (1pt);
\draw[fill=black] (0,1) circle (1pt);
\draw[fill=black] (-0.4,-0.4) circle (1pt);
\draw[-Stealth] (0,0) -- (0.5,0);
\draw[-Stealth] (1,0) -- (1,0.5);
\draw[-Stealth] (1,1) -- (0.5,1);
\draw[-Stealth] (0,1) -- (0,0.5);
\draw[-Stealth] (-0.4,-0.4) -- (-0.15, -0.15);
\draw (0.5,0) -- (1,0);
\draw (1,0.5) -- (1,1);
\draw (0.5,1) -- (0,1);
\draw (0,0.5) -- (0,0);
\draw (-0.15, -0.15) -- (0,0);

\draw[fill=black] (3,0) circle (1pt);
\draw[fill=black] (4,0) circle (1pt);
\draw[fill=black] (4,1) circle (1pt);
\draw[fill=black] (3,1) circle (1pt);
\draw[fill=black] (2.6,-0.4) circle (1pt);
\draw[fill=black] (4.4,1.4) circle (1pt);

\draw[-Stealth] (3,0) -- (3.5,0);
\draw[-Stealth] (4,0) -- (4,0.5);
\draw[-Stealth] (4,1) -- (3.5,1);
\draw[-Stealth] (3,1) -- (3,0.5);
\draw[-Stealth] (2.6,-0.4) -- (2.85, -0.15);
\draw[-Stealth] (4,1) -- (4.25, 1.25);
\draw (3.5,0) -- (4,0);
\draw (4,0.5) -- (4,1);
\draw (3.5,1) -- (3,1);
\draw (3,0.5) -- (3,0);
\draw (2.85, -0.15) -- (3,0);
\draw (4.25, 1.25) -- (4.4,1.4);

\draw[fill=black] (6,0) circle (1pt);
\draw[fill=black] (7,0) circle (1pt);
\draw[fill=black] (7,1) circle (1pt);
\draw[fill=black] (6,1) circle (1pt);

\draw[-Stealth] (6,0) -- (6.5,0);
\draw[-Stealth] (7,0) -- (7,0.5);
\draw[arrows = {-Stealth[reversed]}] (7,1) -- (6.5,1);
\draw[-Stealth] (6,1) -- (6,0.5);
\draw (6.5,0) -- (7,0);
\draw (7,0.5) -- (7,1);
\draw (6.5,1) -- (6,1);
\draw (6,0.5) -- (6,0);
\end{tikzpicture}

\end{figure}

\begin{definition}
  For any combinatorial closed geodesic $\gamma$, we denote by $\gamma^\sharp$
  the unique primitive combinatorial closed geodesic so that $\gamma$ is a power of $\gamma^\sharp$. Recall that $|\gamma^\sharp|$ denotes the length of $\gamma^\sharp$.
\end{definition}

\begin{lemma}
\label{prop:trace}
For each $k = 1, 2, \dots$, it holds
\begin{equation}
\label{eq:transfer}
\tr T^k = 
\sum_{|\gamma| = k} \varepsilon_\gamma |\gamma^\sharp|,
\end{equation}
where the sum runs over all combinatorial closed geodesics (not necessary primitive).
\end{lemma}
\begin{proof}
By the rules for the product of matrices, for $\sigma \in  \beta_{n-1}$ the
diagonal coefficient $(T^k)_{\sigma,\sigma}$ of the k-th power of $T$ at
$\sigma$ is given as
$$
(T^k)_{\sigma, \sigma}= \sum_{\sigma_1, \ldots, \sigma_{k-1}} T_{\sigma, \sigma_1}T_{\sigma_1, \sigma_2} \ldots T_{\sigma_{k-1}, \sigma}
$$
where a priori the $\sigma_j$ run through the basis $\beta_{n-1}$. But of
course, we can restrict to those tuples where none of the terms
$T_{\sigma_j,\sigma_{j+1}}$ is zero. By \cref{prop:Lapl}, this means
  that the sum is over all tuples forming an admissible path
  $(\sigma,\sigma_1,\dots,\sigma_{k-1},\sigma)$. Moreover, by 
  \cref{prop:Lapl} and \cref{def:reversing_ind} of the reversing
  index $\varepsilon_\gamma$, we have 
$$
T_{\sigma, \sigma_1}T_{\sigma_1, \sigma_2} \ldots T_{\sigma_{k-1}, \sigma} = \varepsilon_\gamma
$$
whenever $\gamma$ is represented by the combinatorial closed geodesic
$(\sigma, \sigma_1, \dots, \sigma_{k-1})$. Finally, each combinatorial closed geodesic
$\gamma$ will appear exactly $|\gamma^\sharp|$ times in $\tr T^k$, which
concludes the proof.
\end{proof}

With this trace formula in hand, let us prove \cref{theo:mainfried}, which we recall here. The transfer matrix $T$ is defined in \cref{prop:Lapl}.
\begin{theorem}\label{theo:fried}
Let $\rho(T)$ be the spectral radius of the transfer matrix $T$. Then for $|z|<\rho(T)^{-1}$ we have
\begin{equation}\label{eq:zetaprod}
\zeta_{\mathscr T}(z) = \prod_{\gamma \in \mathcal{P}}\left(1-\varepsilon_\gamma z^{|\gamma|}\right) = \det\left(\mathrm{Id} - z T\right).
\end{equation}
In particular, $\zeta_{\mathscr T}(z)$ extends to a polynomial function
defined on the whole complex plane $\CC$. 
Moreover, it has a zero of order $b_1(M)$ at $z=(n+2)^{-1}$.
\end{theorem}

\begin{proof}
Using that for a positive complex matrix $\log(\det(A))=\tr(\log(A))$ and
using the power series of $\log$, we have
$$
\begin{aligned}
\det(\mathrm{Id} - zT) &= \exp \left(- \sum_{k=1}^\infty \frac{z^k}{k} \tr T^k
\right)&= \exp\left( - \sum_{k = 1}^\infty \frac{z^k}{k} \sum_{|\gamma| = k} \varepsilon_\gamma |\gamma^\sharp|\right).
\end{aligned}
$$
This makes sense for $z\in\CC$ with $|z|<\rho(T)^{-1}$ because then $\lVert
zT\rVert <1$, hence  $\mathrm{Id}-zT$ is positive and the series converge absolutely. 

Now we have
\begin{equation*}
{\sum_{|\gamma| = k} \varepsilon_\gamma |\gamma^\sharp| = \sum_{\gamma \in \mathcal P} |\gamma|\sum_{p\in\mathbb{N}:~p\cdot|\gamma| = k}\varepsilon_\gamma^p},
\end{equation*}
hence one gets
$$
\begin{aligned}
\det(\mathrm{Id} - zT) &= \exp \left(- \sum_{\gamma \in \mathcal P} |\gamma| \sum_{p=1}^\infty\frac{z^{p|\gamma|}}{p|\gamma|} \varepsilon_\gamma^p\right) \\
&= \prod_{\gamma \in \mathcal P} \exp \left(- \sum_p \frac{\left(\varepsilon_\gamma z^{|\gamma|}\right)^p}{p}\right),
\end{aligned}
$$
which proves \cref{eq:zetaprod}. Finally, by \cref{prop:Lapl}, we have
\begin{equation*}
  \begin{split}
    \det\left(\mathrm{Id}-(z-\frac{1}{2+n})T\right) &= \det\left( \mathrm{Id}
                                                      -(z-\frac{1}{2+n})
                                                      ((n+2)\mathrm{Id}-\Delta)\right) \\
    &=  \det\left(\frac{1}{2+n}\Delta+ z (\Delta-(n+2))  \right)
  \end{split}
\end{equation*}
which vanishes at zero of order $\dim\ker(\Delta_{n-1})$, as we see by
diagonalizing $\Delta_{n-1}$.
Moreover, by a standard linear algebra/functional analysis result (sometimes
called finite 
dimensional Hodge theory, see~\cite[Appendix A]{nicolaescu2008reidemeister} for an exposition of this circle of ideas), we have
$$\ker(\partial) = \im(\partial)\oplus \ker(\Delta)\quad\implies \;\dim \ker \Delta_{n-1} = b_{n-1}(M).$$
Finally by Poincar\'e duality $b_{n-1}(M)=b_1(M)$ and we conclude that
$\det(\mathrm{Id} - zT)$ vanishes of order $b_1(M)$ at $z = (n+2)^{-1}$.
\end{proof}

We now give the short proof of \cref{corol:2_and_4_mf}.
\begin{proof}[Proof of \cref{corol:2_and_4_mf}]
  For a compact connected orientable surface $F$, the Euler characteristic
  satisfies $\chi(F)=2-b_1(F)$ and we have
  $n=2$, hence by \cref{corol:1stBetti} the combinatorial closed geodesics of length
  bounded by the number of edges in the triangulation and their reversing
  indices determines $\chi(F)$.

  If $M$ is a compact connected oriented manifold with $\dim(M)\leqslant 4$ we know
  a priori that $b_0(M)=b_4(M)=1$ and by \cref{corol:1stBetti} we determine
  $b_1(M)=b_3(M)$. Finally, $b_2(M)=\chi(M)-2+b_1(M)+b_3(M)$ is now determined
  by the Euler characteristic which can be read of from the combinatorial data
  of the triangulation (namely the number of simplices of different dimension).
\end{proof}
\subsection{The non-compact case: $L^2$-Betti numbers}
In this section we consider the setting of \cref{subsec:L2}: the compact
manifold $M$ with triangulation $\mathscr T$ comes with a normal covering
$\wh{M}$ with free action by a quotient $\pi$ of the fundamental group $\pi_1(M)$. The triangulation $\mathscr T$ lifts as $\wh{\mathscr{T}}$, for which we fix a basis $\{\hat{\sigma}^{n-1}_i\}_{i=1, \ldots |\mathscr T^{(n-1)}|}$ for the space $C_{n-1}(\wh{\mathscr{T}})$ as $\ZZ[\pi]$-module.

We denote by $\wh{\mathcal P}$ the set of primitive combinatorial closed geodesics in $\wh{\mathscr{T}}$ starting from one of the $\hat{\sigma}^{n-1}_i$.
Recall that $T$ is the transfer operator associated to the geodesic random walk on $\wh{\mathscr T}$ described in \cref{prop:Lapl}.

Then \cref{prop:trace} holds true in this setting, namely:
\begin{lemma}
\label{prop:L2trace}
For each $k = 1, 2, \dots$, it holds
\begin{equation}
\label{eq:transferL2}
\tr_{\mathrm{vN}} T^k = 
\sum_{|\gamma| = k} \varepsilon_\gamma |\gamma^\sharp|,
\end{equation}
where the sum runs over all combinatorial closed geodesics (not necessary primitive) which start at one of the $\hat{\sigma}^{n-1}_i$.
\end{lemma}
\begin{proof}
Using the definition of the von Neumann trace given in \cref{eq:L2tracedef}, the proof is exactly the same as for \cref{prop:trace}.
\end{proof}

Now we prove \cref{theo:mainL2}, which we recall here:
\begin{theorem}
The function 
\begin{equation}
\label{zeta}
\zeta^{(2)}_{\wh{\mathscr T}}(z) = \prod_{\gamma \in \widehat{\mathcal P}} (1- \varepsilon_\gamma z^{|\gamma|})
\end{equation}
converges for $|z| \ll 1$, and has an analytic extension to the disk of diameter $(0, \frac 1 {n+2})$. Moreover,
$$
\zeta^{(2)}_{\wh{\mathscr T}}\left(\frac{1}{n+2}-z\right) = z^{b_1^{(2)}(M,\pi)} f(z) 
$$
with a function $f$ which is continuous at $0$.  If $\Delta_{n-1}^{(2)}$ is of
determinant class, then
\begin{equation*}
  f(0) =(n+2)^{2b_1^{(2)}(M,\pi)-|\mathscr T^{(n-1)}|} \cdot \detFK(\Delta_{n-1}^{(2)}).
\end{equation*}
If $\Delta_{n-1}^{(2)}$ is not of determinant class, then $f(0)=0$ but $f$
converges slower to $0$ than any power of $z$ in the sense that for all $C>0$,
$\alpha>0$ there is $\varepsilon>0$ such that
\begin{equation*}
  f(z)> Cz^\alpha\qquad\forall 0<z<\varepsilon.
\end{equation*}
In particular, the function $\zeta^{(2)}_{\wh{\mathscr T}}$ and hence the
combinatorial closed geodesics on $\widehat M$ together with the data of
$|\mathscr T^{(n-1)}|$ determine $b_{n-1}^{(2)}(M, \pi)$.
\end{theorem}

\begin{remark}
  The term $\det_{FK}(\Delta^{(2)}_{n-1})$ appearing in \cref{theo:mainL2} is
  a very delicate spectral invariant which highly depends on the specific
  triangulation and is not a topological invariant (only the property of being
  of determinant class is a topological invariant, even a homotopy invariant,
  of the manifold $M$). We therefore don't assign too much meaning to it.

  Note, however, that a combination of the Fuglede--Kadison determinants of
  all the combinatorial $L^2$-Laplacians gives the $L^2$-torsion of $M$, a
  very interesting topological invariant, compare \cite[Chapter 3]{Lueck}. This is similar to the compact Riemannian case, it is well--known that the Ray--Singer zeta determinant of the Laplacian acting on functions is not a topological invariant. However, a combination of the zeta determinants for the Laplacian acting on forms of all degrees gives the analytic torsion which happens to be a topological invariant.   
   A
  case of interest to us are closed 3-manifolds. Here,
  the $L^2$-torsion (of the universal covering) is proportional to the sum of
  the volumes of the hyperbolic pieces in the JSJ-decomposition by
  \cite{LueckSchick}. 
\end{remark}

\begin{proof}
Since the triangulation $\wh{\mathscr T}$ is a lift of a triangulation of a
compact manifold, there is a global upper-bound $K$ on the number of neighbors
of any $n-1$ simplex, in particular there are less than $K^k$ primitive
combinatorial closed geodesics of length $k$. It follows that the infinite product \eqref{zeta} converges for $|z| < K^{-1}$.

Now using \cref{prop:L2trace} and functional calculus one sees that 
\begin{equation}
\label{eq:logdet}
{\det}_{\mathrm{FK}}(\Id - zT) = \exp\left(- \sum_{k=1}^\infty \frac{z^k}{k} \tr_{\vN} T^k
\right)
\end{equation}
on the real interval $\left]-\frac 1 K, \frac 1 K\right[$, since there the
operator $\log(\Id - zT)$ is well-defined and is given by the power series in
\cref{eq:logdet}. One deduces just as in the proof of \cref{theo:fried} that
$$\zeta^{(2)}_{\wh{\mathscr T}}(z) = {\det}_{\mathrm{FK}}(\Id - z T).$$

We will use \cref{item:add} for the orthogonal decomposition $\Id = p_K \oplus q_K$
with $p_K$ the orthogonal projection onto $\ker(\Delta^{(2)}_{n-1})$ and
$q_K=1-p_K$. This decomposition is preserved by $\Delta^{(2)}_{n-1}$ and
satisfies of course $\Delta_{n-1}^{(2)}p_K=0$, and that
$\Delta^{(2)}_{n-1}q_K$ is injective on $\im(q_K)$. Note also that by
definition, one has 
$$
\dim_{\vN}(\im(p_K)) = b_{n-1}^{(2)}(M,\pi) =
b_1^{(2)}(M,\pi).
$$

Note that from \cref{prop:Lapl} one has
\begin{equation*}
  T = (n+2)\Id - \Delta_{n-1}^{(2)},
\end{equation*}
hence we obtain
\begin{equation*}
  \begin{split}
    \zeta^{(2)}_{\wh{\mathscr T}}\Bigl(&\frac{1}{n+2}-z\Bigr) =
                              {\det}_{\mathrm{FK}}\left(\Id-(\frac{1}{n+2}-z)T\right)\\
    &= {\det}_{\mathrm{FK}}\left(\Id-\left(\frac{1}{n+2}-z\right)
      \left((n+2)\Id-\Delta_{n-1}^{(2)}\right)\right) \\
  &=
    {\det}_{\mathrm{FK}}\left(\left(\left(\frac{1}{n+2}-z\right)\Delta_{n-1}^{(2)}+z(n+2)\Id\right)(p_K\oplus
    q_K)\right)\\
    &\stackrel{\eqref{item:add}}{=} {\det}_{\mathrm{FK}}(z(n+2)p_K) \cdot {\det}_{\mathrm{FK}}\left(\left((\frac{1}{n+2}-z)\Delta_{n-1}^{(2)}+z(n+2)\Id\right)q_K\right)\\
     &\stackrel{\eqref{item:mult}}{=} \left((n+2)z\right)^{b_1^{(2)}(M,\pi)}
    \left(\frac{1}{n+2}-z\right)^c\cdot
    \underbrace{{\det}_{\mathrm{FK}}\left(\Delta^{(2)}_{n-1}-\frac{z(n-2)}{(n+2)^{-1}-z}\Id\right)}_{\xrightarrow{z\to
    0} {\det}_{\mathrm{FK}}(\Delta_{n-1}^{(2)})}\\
&= z^{b_1^{(2)}(M,\pi)} \cdot f(z) 
  \end{split}
\end{equation*}
Here $c={\dim_{vN}(\im(q_K))}=|\mathscr{T}_{n-1}|-b_1^{(2)}(M,\pi) \geqslant 0$ is
  an irrelevant non-negative term and 
\begin{equation*}
  f(z) = (n+2)^{b_1^{(2)}(M,\pi)}\left(\frac{1}{n+2}-z\right)^c\cdot{\det}_{\mathrm{FK}}\left(\Delta^{(2)}_{n-1}-\frac{z(n-2)}{(n+2)^{-1}-z}\Id\right)
  \end{equation*}
   is continuous on the interval $[0, \frac{1}{n+2})$ with
   \begin{equation*}
     f(0)= (n+2)^{b_1^{(2)}(M,\pi) - (|\mathscr{T}_{n-1}|- b_1^{(2)} (M,\pi))} \cdot
     {\det}_{\mathrm{FK}}(\Delta^{(2)}_{n-1})
   \end{equation*}
   by \cref{item:conv}.

  By \cref{lem:FK}, the vanishing order of $f$ is zero even if
  $f(0)=0$ and hence the vanishing order of $\zeta^{(2)}_{\wh{\mathscr{T}}}$ at
  $\frac{1}{n+2}$ is precisely $b_1^{(2)}(M,\pi)$ which is hence determined by
  $\zeta^{(2)}_{\wh{\mathscr{T}}}$ and consequently by the combinatorial closed
  geodesics and their reversing indices.
\end{proof}

\section{Combinatorial linking number}
\label{sec:link}
In this section we prove \cref{theo:mainlink}. We take a compact oriented
3-manifold~$M$ with triangulation $\mathscr T$ and we let $\mathscr T^\vee$ be
a dual {polyhedral} decomposition of $\mathscr T$. Take two oriented knots
$\kappa_1 \in C_1(\mathscr T)$ and $\kappa_2 \in C_1(\mathscr T^\vee)$, which
we assume to be rationally homologically trivial in $M$, in the sense that
there are positive integers $p_1, p_2$ and $2$-dimensional chains $\sigma_1
\in C_2(\mathscr T)$ and $\sigma_2 \in C_2(\mathscr T^\vee)$ such that 
$$
\partial \sigma_j =p_j \kappa_j, \quad j = 1,2.
$$

Our definition of ``knot'' is very flexible, any closed (integral) $1$-chain $\kappa\in
C_1(\mathscr T)$ is permitted.

Recall e.g.~from \cite[Section 2.2]{Lescop_book} that the linking number of $\kappa_1$ and $\kappa_2$ is defined as the algebraic intersection number of $\sigma_1$ with $\kappa_2$ divided by $p_1$, which can be written  
\begin{equation}\label{eq:deflk}
\lk(\kappa_1, \kappa_2) =
\frac 1 {p_2} \langle \kappa_1, \star \sigma_2 \rangle \in \mathbb Q.
\end{equation}

Our aim is to compute this quantity with combinatorial means. 

\begin{definition}\label{def:orthogeodesic}
  We define $\mathcal G^\perp(\kappa_1, \kappa_2)$ {to be} the set of
  {\emph{orthogeodesic paths}} from $\kappa_1$ to $\kappa_2$,
  i.e.~combinatorial geodesic paths $c = (\tau_1, \ldots, \tau_k)$ (in the
  sense of \cref{def:geodesic}) in the $2$-skeleton of $\mathscr T$ such that
$$
|\langle \partial \tau_1, \kappa_1\rangle| > 0 \quad \text{ and } \quad
|\langle \tau_k, \star \kappa_2 \rangle | >0.$$ In other words, $c$ starts in
$\partial^\star \kappa_1$ and ends up in $\star \kappa _2$, see also
\cref{fig:link}.
\end{definition}

Again, we will denote by $|c|$ the length of $c$, and for each orthogeodesic path $c = (\tau_1, \dots, \tau_k) \in \mathcal G^\perp(\kappa_1, \kappa_2)$, we define the \textit{incidence number} of $c$ on $(\kappa_1, \kappa_2)$ as
$$
m_c = \langle \partial \tau_1, \kappa_1\rangle \, 
\langle \tau_k, \star \kappa_2 \rangle.
$$
If each knot $\kappa_j$ is \textit{simple}, in the sense that the boundary of
each $2$-simplex of $\mathscr T$ occurs  at most once in $\kappa_1$ and each
$1$-cell of $\mathscr T^\vee$ occurs in $\kappa_2$ at most once, then $m_c \in
\{-1, 1\}$. Although it will not be used in
the current paper, note that one can always find a triangulation $\mathscr T$ so that $\kappa_1$ (resp. $\kappa_2$) is homologous to a simple knot in $C_1(\mathscr T)$ (resp. $C_1(\mathscr T^\vee)$). Hence the incidence number can be thought as a sign.

Recall from \cref{def:reversing_ind}  the reversing index $\varepsilon_c =
(-1)^{n_c}$ of $c$. We have the following trace formula, analogous to \cref{prop:trace}, which describes how the entries of the transfer matrix $T$ {count} orthogeodesic paths.
\begin{lemma}
\label{prop:link}
For any $k \in \ZZ_{\geqslant 1}$, it holds
$$\langle T^{k-1} \partial^\star \kappa_1, \star \kappa_2 \rangle = \sum_{\substack{c \in \mathcal G^\perp(\kappa_1, \kappa_2) \\ |c| = k}} \varepsilon_c m_c.$$
\end{lemma}
\begin{proof}
For $\sigma, \tau \in  \beta_2$, one has
$$(T^{k-1})_{\sigma, \tau} = \sum_{\sigma_1, \ldots, \sigma_{k-2}} T_{\sigma, \sigma_1} T_{\sigma_1,\sigma_2}\ldots T_{\sigma_{k-2},\tau}.$$
By definition, if $c = (\sigma, \sigma_1, \ldots, \sigma_{k-2}, \tau)$ is an orthogeodesic from $\sigma \in \partial^\star \kappa_1$ to $\tau \in \star \kappa_2$, then
$$T_{\sigma, \sigma_1} T_{\sigma_1,\sigma_2}\ldots T_{\sigma_{k-2},\tau} = \varepsilon_c.$$
Moreover, the sum of contributions of this path $c$ in the scalar
product is exactly $m_c$, and it proves the lemma.
\end{proof}

We are now ready to prove \cref{theo:mainlink}, which we recall.
\begin{theorem}
\label{theo:link}
Let $\rho(T)$ be the spectral radius of the transfer matrix $T$. The series
$$
\eta(z) = \sum_{c \in \mathcal G^\perp(\kappa_1, \kappa_2)} \varepsilon_c m_c z^{|c|},$$ 
converges for $|z| <\frac 1 {\rho(T)}$. It defines a rational function of $z$, which is regular at $z=1/(n+2)$ with
$$
\eta\left(\frac 1{n+2}\right) = \lk(\kappa_1, \kappa_2).
$$
\end{theorem}

\begin{proof}
Finite dimensional Hodge theory gives
$$C_1(\mathscr T) = \im \partial \oplus \im \partial^\star \oplus \ker \Delta.$$
Moreover, $\Delta$ preserves this decomposition, hence it maps $\im( \partial)$
(resp. $\im( \partial^\star)$) to itself isomorphically. We denote by $K_1\colon
\im(\partial) \to \im(\partial)$ the operator which is the inverse of
$\Delta_1$ on $\im( \partial)\subset C_1(\mathscr T)$.
Similarly, we let $K_2\colon \im(\partial^\star)\to \im(\partial^\star)$ be
the inverse of $\Delta_2$ on $\im(\partial^\star)\subset C_2(\mathscr T)$.
Note that $\Delta_2\partial^\star=\partial^*\Delta_1$ and hence also
$K_2\partial^\star = \partial^\star K_1$.
Now, since $p_1\kappa_1 \in \im \partial$ and $\Delta_1$ is invertible on $\im \partial$,  $K_1\kappa_1$ is well--defined and one has
$$
\partial \partial^\star K_1 \kappa_1 = \Delta_1K_1\kappa_1= \kappa_1,
$$
and thus by \cref{eq:deflk}, one gets
\begin{equation}\label{eq:formulalink}
  \begin{split}
    \lk(\kappa_1, \kappa_2) &=\frac 1 {p_2} \langle \partial \partial^\star K_1
                              p_1 \kappa_1, \star\sigma_2\rangle
     = \frac{1}{p_2} \langle \partial^\star
      K_1\kappa_1,\star\partial \star\star\sigma_2\rangle \\
    &= \frac{1}{p_2} \langle \partial^\star K_1\kappa_1,\star p_2\kappa_2\rangle
      = \langle\partial^\star K_1\kappa_1,\kappa_2\rangle.
  \end{split}
\end{equation}
Now if $\frac 1 z$ is not in the spectrum of $T$, we have
\begin{equation}
\label{1}
\left\langle  \left( \frac 1 z \Id - T\right)^{-1}
 \partial^\star \kappa_1, \star\kappa_2\right\rangle = \left\langle \left(\Delta_2 + \left(\frac 1 z - (n+2)\right) \Id\right)^{-1}  \partial^\star \kappa_1, \star\kappa_2\right\rangle 
\end{equation}
by \cref{prop:Lapl}. For $|z| <\frac 1 {\rho(T)}$ we may expand \cref{1} to get
$$
 \left\langle \left(\Delta_2 + \left(\frac 1 z - (n+2)\right)
     \Id\right)^{-1} \partial^\star \kappa_1, \star\kappa_2\right\rangle = \sum_{k=1}^\infty z^{k} \langle T^{k-1}\partial^\star  \kappa_1, \star\kappa_2 \rangle = \eta(z)
$$
where the last equality comes from  \cref{prop:link}. This shows that $\eta(z)$ is a rational function in $z$. Finally, note that $\Delta_2 + \omega$ is invertible on $\im(\partial^\star)$ for $|\omega|$ small, with
$$
(\Delta_2 + \omega)^{-1} = K_2 - \omega(\Delta_2 + \omega)^{-1}K_2.
$$
Applying this with  $\omega = \frac 1 z - (n+2)$, evaluating at $z=\frac{1}{n+2}$, we obtain 
$$\eta\left(\frac 1{n+2}\right) = \langle
K_2\partial^\star\kappa_1,\star\kappa_2\rangle = \langle \partial^\star K_1 \kappa_1, \kappa_2\rangle = \lk(\kappa_1, \kappa_2\rangle,$$ which concludes the proof.
\end{proof}

\begin{remark}
The operator $\partial^*K_1$ is the discrete Hodge theoretic version of
linking forms appearing in numerous articles like \cite{vogel},
\cite{harris2004}. The idea of using the Laplacian to produce linking forms
probably goes back to the Gauss integral \cite[Definition
15.4.1]{dubrovin2012modern} who already used the Schwartz kernel of
$d^*\Delta^{-1}$ to compute linking numbers. The dynamical version of linking
forms which inspired our approach comes from~\cite{Dang_Riviere} and
also~\cite[
Section 12]{harvey2001finite}. Most relevant to our discussion is the work by
Delsarte~\cite{delsarte} and Huber~\cite{huber1956neue,huber1959analytischen}
who were able to relate Poincar\'e series on surfaces of constant negative
curvature and the Laplacian. This allowed these authors to prove the analytic
continuation exactly in the same spirit as Selberg's work relating zeta functions to the Laplacian.   
\end{remark}

\section{Examples and final remarks}

In this section, we briefly indicate some example situations for our main
results. Unfortunately, triangulations of manifolds (even when the
generalization of \cref{rem:generalize_trian} is taken into account) turn out to be quite complicated even for simple manifolds,
therefore we will not compute enough terms of the zeta functions to actually
derive interesting consequences. But we also would like to point out that we do not know any surface of constant negative curvature where the length of closed geodesics is known so in fact the Selberg zeta function cannot be explicitely computed.

There is one (simple) exception:
\begin{example}
  Consider $\partial\Delta^{n+1}$, the boundary of the standard $n+1$-simplex
  as triangulation $\mathscr T_{\partial\Delta^{n+1}}$ of the $n$-sphere $S^n$.

  Here, any two $n-1$-simplices which have a common face bound a common
  $n$-simplex. Consequently, there are no combinatorial geodesic paths and
  loops at all, and $\zeta_{\mathscr T_{\partial\Delta^{n+1}}}(z)=1$ is the
  constant function $1$.

  The vanishing order at $(n+2)^{-1}$ hence is $0$, compatible with the fact
  that $b_1(S^n)=0$ for $n>1$.
\end{example}

\begin{example}
  Consider the quasi-triangulation of the $2$-torus given by the following
  \cref{fig:torus}. Note that, as usual, we have to identify the top and
  bottom segment as well as the left and right segment.

  Note that this is not quite a triangulation as any pair of 2-simplices which
  does not share an edge does intersect in two distinct vertices.

  The shortest combinatorial closed geodesics are of length $2$. An inspection
  shows that these are obtained as the straight horizontal, vertical and
  diagonal lines in the picture. One example is the visible diagonal drawn in
  \cref{fig:torus} in blue. Passing through them in the opposite
  direction here is a cyclic permutation (because of length $2$), hence we
  have 6 combinatorial closed geodesics of length $2$. For each of them, the
  reversing index is $1$.

  The contribution of these shortest closed geodesics to the zeta function is
  therefore
  \begin{equation*}
    (1-z^2)^6.
  \end{equation*}

There are no combinatorial closed geodesics of length $3$, the candidates drawn
in \cref{fig:torus} in red or green are not permitted as the consecutive horizontal and
vertical edge bound the same triangle (the bottom right one or the top right one).

But there are many combinatorial geodesic paths of length 4 and it would be
challenging to make a precise list of them.
  
\begin{figure}[h]
\caption{\label{fig:torus} A permitted $\Delta$-complex decomposition of the
  torus which is not quite a triangulation. The top and bottom segment have to
  be identified as well as the left and right one.}
  \vspace{0.2cm}

 \begin{tikzpicture}
   \draw[fill=black] (0,0) circle (   1pt);
   \draw[fill=black] (1,0) circle (1pt);
   \draw[fill=black] (1,1) circle (1pt);
   \draw[fill=black] (0,1) circle (1pt);
   \draw[fill=black] (0,2) circle (   1pt);
   \draw[fill=black] (1,2) circle (   1pt);
   \draw[fill=black] (2,0) circle (1pt);
   \draw[fill=black] (2,1) circle (1pt);
   \draw[fill=black] (2,2) circle (1pt);
   \draw[teal,-Stealth] (0,0) -- (0.5,0);
   \draw[teal] (0.5,0) -- (1,0);
   \draw[-Stealth] (1,0) -- (1.5,0);
   \draw (1.5,0) -- (2,0);
   \draw[-Stealth] (0,1) -- (0.5,1);
   \draw (0.5,1) -- (1,1);
   \draw[-Stealth] (1,1) -- (1.5,1);
   \draw (1.5,1) -- (2,1);
   \draw[-Stealth] (0,2) -- (0.5,2);
   \draw (0.5,2) -- (1,2);
   \draw[purple,-Stealth] (1,2) -- (1.5,2);
   \draw[purple] (1.5,2) -- (2,2);
   \draw[-Stealth] (1,0) -- (1.5,0);
   \draw (1.5,0) -- (2,0);
\draw[purple,-Stealth] (0,0) -- (0,0.5); \draw[purple] (0,0.5) -- (0,1);
\draw[-Stealth] (0,1) -- (0,1.5); \draw (0,1.5) -- (0,2);
\draw[-Stealth] (1,0) -- (1,0.5); \draw (1,0.5) -- (1,1);
\draw[-Stealth] (1,1) -- (1,1.5); \draw (1,1.5) -- (1,2);
\draw[-Stealth] (2,0) -- (2,0.5); \draw (2,0.5) -- (2,1);
\draw[teal,-Stealth] (2,1) -- (2,1.5); \draw[teal] (2,1.5) -- (2,2);

   \draw[orange,-Stealth] (0,0) -- (0.5,0.5);
   \draw[orange] (0.5,0.5) -- (1,1);
   \draw[teal,-Stealth] (1,0) -- (1.5,0.5);
   \draw[teal] (1.5,0.5) -- (2,1);
   \draw[purple,-Stealth] (0,1) -- (0.5,1.5);
   \draw[purple] (0.5,1.5) -- (1,2);
   \draw[orange,-Stealth] (1,1) -- (1.5,1.5);
   \draw[orange] (1.5,1.5) -- (2,2);

 \end{tikzpicture}
\end{figure}

The given quasi-triangulation of the torus of course lifts to a triangulation
of the universal covering $\RR^2$. The picture is obtained from
\cref{fig:torus} by periodically repeating in all directions.

The shortest combinatorial closed geodesics here have length $6$, one example
is given by concatenating the red and the green geodesics in
\cref{fig:torus}. Here, the reversing number is $2$, hence the reversing index
again is $1$. 
\end{example}

\subsection{Concluding remarks}

The fundamental identity of our work is \cref{eq:transfer_eq} which relates
the combinatorial Laplacian to the transfer operator of a random walk. In our
case, this is a signed random walk on the $n-1$-skeleton of a manifold.

As mentioned earlier, a similar approach has been used on the $0$-skeleton
(and dually also $1$-skeleton) of a regular graph (which could arise as the
$1$-skeleton of a CW-complex). There, it gives rise to an
honest (non-signed) random walk on the graph.

The key point of our approach is the packaging of the combinatorial dynamical
information in a zeta function. This encodes the spectral information in a
different way from the approach of Varopoulos
\cite{Varopoulos}. Varopoulos' result is more direct, but he uses significantly
the positivity of the random walk, which is not available in our context.

In the case of graphs and more specifically trees with a cocompact action
of a group $\pi$, there are also approaches to define and study zeta functions
and relate them to the combinatorial Laplacian. A specific example here is the
Ihara zeta function as introduced and studied by Bass in \cite{Bass92}. On a
formal level, many fundamental properties are similar to ours, compare
e.g.~\cite[Theorem 3.9]{Bass92}. On the other hand, there are also many
fundamental differences: Bass uses non-commutative determinants in the style
of Hattori--Stallings. They are finer, but much more intricate than our
Fuglede--Kadison invariant, whereas the latter can be defined and manipulated
in more general situations. Bass deals with proper actions which are not
necessarily free (free actions on trees occur only by free groups).

This leaves an interesting question in our context: can one prove suitable
generalizations of Theorem \ref{theo:mainL2} when we have a simplicial action
on a non-compact manifold $\wh M$ which is proper and cocompact, but not free
(i.e.~with finite stabilizers of simplices)? A suitable theory of
$L^2$-invariants in this context exists, see \cite[Section 6]{Lueck}.

Another question which we leave open is inspired by the work of Varopoulos
\cite{Varopoulos} and H\"opfner \cite{Hoepfner} which both get
informations about the Novikov-Shubin invariant of the manifold from the
(signed) random walk. It seems plausible the Novikov-Shubin invariant of
$\Delta_1^{(2)}$ is encoded in the behavior of the
``secondary function'' $f(z)$ of \cref{theo:mainL2} which describes the deviation
from the power function $z^{b_1^{(2)}(M,\pi)}$ of the $L^2$-zeta function near
$\frac{1}{n+2}$.

\bibliographystyle{plain} 
\bibliography{biblio} 

 

\end{document}